\documentclass[oneside,english]{amsart}
\usepackage[T1]{fontenc}
\usepackage[latin9]{inputenc}
\usepackage{amstext}
\usepackage{amsthm}
\usepackage{amssymb}
\usepackage{cancel}
\usepackage{graphicx}

\usepackage{amsmath}

\makeatletter
\numberwithin{equation}{section}
\numberwithin{figure}{section}
\theoremstyle{plain}
\newtheorem{thm}{\protect\theoremname}[section]
\theoremstyle{definition}
\newtheorem{defn}[thm]{\protect\definitionname}
\theoremstyle{remark}
\newtheorem*{rem*}{\protect\remarkname}
\theoremstyle{definition}
\newtheorem{example}[thm]{\protect\examplename}

\DeclareMathOperator{\sech}{sech}

\makeatother

\usepackage{babel}
\providecommand{\definitionname}{Definition}
\providecommand{\examplename}{Example}
\providecommand{\remarkname}{Remark}
\providecommand{\theoremname}{Theorem}

\begin{document}
\title[Random Walks for Bernoulli and Euler Polynomials]{Random Walk Models for Nontrivial Identities of Bernoulli and Euler
Polynomials}
\author{Lin Jiu}
\address{Zu Chongzhi Center for Mathematics and Computational Sciences, Duke
Kunshan University, Kunshan, Suzhou, Jiangsu Province, PR China, 215316.}
\email{lin.jiu@dukekunshan.edu.cn}
\author{Italo Simonelli}
\address{Zu Chongzhi Center for Mathematics and Computational Sciences, Duke
Kunshan University, Kunshan, Suzhou, Jiangsu Province, PR China, 215316.}
\email{italo.simonelli@dukekunshan.edu.cn}
\author{Heng Yue*}
\address{Class of 2023, Duke
Kunshan University, Kunshan, Suzhou, Jiangsu Province, PR China, 215316.}
\email{heng.yue@dukekunshan.edu.cn}

\thanks{*corresponding author}

\begin{abstract}
We consider the $1$-dimensional reflected Brownian motion and $3$-dimensional
Bessel process and the general models. By decomposing the hitting
times of consecutive sites into loops, we obtain identities, called
loop identities, for the generating functions of the hitting times. After proving 
this decomposition both combinatorially and inductively, we consider the case that 
 sites are equally distributed. Then, from loop identities,
we derive expressions of Bernoulli and Euler polynomials, in terms
of Euler polynomials of higher-orders. 
\end{abstract}


\maketitle

\section{Introduction}

Random walks, of various types, have been comprehensively studied
and widely applied in physics, engineering, and especially many fields
of mathematics. The connection between random walk models and special
polynomials, especially Bernoulli and Euler polynomials, was not obvious,
since those polynomials, though with numerous applications in different
areas, mainly appear in number theory and combinatorics. However,
the recent work by the first author and Vignat \cite{JiuVignat}, on
the $1$-dimensional reflected Brownian motion and $3$-dimensional
Bessel process, discovered and proved that some non-trivial identities involving Bernoulli
and Euler polynomials of order $p$, denoted by $B_{n}^{(p)}(x)$
and $E_{n}^{(p)}(x)$ and defined via their exponential generating
functions 
\begin{equation}
\left(\frac{t}{e^{t}-1}\right)^{p}e^{xt}=\sum_{n=0}^{\infty}B_{n}^{(p)}(x)\frac{t^{n}}{n!}\quad\text{and}\quad\left(\frac{2}{e^{t}+1}\right)^{p}e^{xt}=\sum_{n=0}^{\infty}E_{n}^{(p)}(x)\frac{t^{n}}{n!}.\label{eq:GF}
\end{equation}
In particular, $B_{n}(x)=B_{n}^{(1)}(x)$ and $E_{n}(x)=E_{n}^{(1)}(x)$
are the ordinary Bernoulli and Euler polynomials; and Bernoulli numbers
$B_{n}=B_{n}(1)$ and Euler numbers $E_{n}=2^{n}E_{n}(1/2)$ are special
evaluations. See, e.g., \cite[Chap.~24]{NIST} for details. 

This study originally arises from early work \cite[Eq.~(3.8)]{Euler},
where the first author, Moll, and Vignat expressed the usual Euler
polynomials as a linear combination of higher-order Euler polynomials:
for any positive integer $N$, 
\[
E_{n}(x)=\frac{1}{N^{n}}\sum_{\ell=N}^{\infty}p_{\ell}^{(N)}E_{n}^{(\ell)}\left(\frac{\ell-N}{2}+Nx\right).
\]
It is surprising that the positive coefficients $p_{\ell}^{(N)}$
also appear as transition probabilities in the context of a random
walk over a finite number of sites \cite[Note 4.8]{Euler}, which
reveals the possibility to connect random walks and $E_{n}^{(p)}(x)$,
as well as $B_{n}^{(p)}(x)$. 

Results in \cite{JiuVignat} are obtained by decomposing the successive
hitting times ONLY of \emph{two}, \emph{three}, and \emph{four} fixed
levels, i.e., walks with \emph{one} or \emph{two} loops. Therefore,
it is the purpose of this paper to generalize this work into general
$n$ loops,
\begin{itemize}
\item by both inclusion-exclusion principle and induction;
\item in order to obtain the general hitting time decomposition for $n$ loops;
\item and to derive the corresponding identities involving $B_{n}^{(p)}(x)$ and $E_{n}^{(p)}(x)$. 
\end{itemize}

Hence, this paper is organized as follows. In Section \ref{sec:Loop}, we introduce basic notation for random walk,
especially the generating function of the hitting time; the models of $1$-loop and $2$-loop cases are recalled as examples. In Section \ref{sec:NLoop}, we generate the model to general $n$-loops, with two proofs, both combinatorially, by inclusion-exclusion principle, and inductively. In Section \ref{sec:umbral}, we recall the three umbral symbols: Bernoulli, Euler and uniform symbols, with their important properties and connection. These formulas are crucial to derive identities in Section \ref{sec:1loop}. In this last section, we first consider the $1$-dimensional reflected Brownian motion model; as an analogue, the loop decomposition and identites in $3$-dimensional Bessel process model are also derived. 

\section{Preliminaries: loops}\label{sec:Loop}

One can find similar summary in \cite{JiuVignat}, except for a slight
change in the notation: see Def.~\ref{def:phi} below. We still include
them to make this paper self-contained.

\subsection{Notation for paths and loops}

We begin with some notation on the moment generating functions of
random walks among loops. 
\begin{defn}
\label{def:phi}Consider sites $a<b$ and a third site $c$, different
from $a$ and $b$. 
\begin{itemize}
\item We let $\phi_{a\rightarrow b}$ be the \emph{moment generating function
of the hitting time of site $b$ starting from site $a$}; also let
$\phi_{b\rightarrow a}$ be the \emph{counterpart from $b$ to $a$}.
Therefore, 
\[
L_{a,b}:=\phi_{a\rightarrow b}\phi_{b\rightarrow a}.
\]
is the moment generating function of the hitting time 
\begin{itemize}
\item starting from $a$;
\item hitting $b$ first; 
\item and finally returning to $a$. 
\end{itemize}
\item[] Note the symmetry that $L_{b,a}=\phi_{b\rightarrow a}\phi_{a\rightarrow b}=L_{a,b}$.
Thus, this is the \emph{loop} between sites $a$ and $b$. 
\item Also, we let $\phi_{a\rightarrow b|\cancel{c}}$ be the moment generating
function of the hitting time of site $b$ starting from site $a$
\emph{before hitting site $c$}; and similarly for $\phi_{b\rightarrow a|\cancel{c}}$.
It is easy to see, $\phi_{a\rightarrow b|\cancel{c}}=\phi_{a\rightarrow b}$
if $c>b$ and $\phi_{a\rightarrow b|\cancel{c}}=0$ if $a<c<b$. 
\item If $a$ is the $m$th site, denoted by $a_{m}$; and $b$ is the $n$th
site as $a_{n}$, we shall use $\phi_{m\rightarrow n}$, for simplicity,
instead of $\phi_{a_{m}\rightarrow a_{n}}$. 
\item Finally, let $\phi_{m\rightarrow n|\cancel{k}}$ be the moment generating
function of the hitting time from the $m$th site to the $n$th site
before hitting the $k$th site. 
\item We can similarly use $t_{a\rightarrow b}$ etc.~for the hitting times, rather than
their moment generating functions. For instance,  $t_{m\rightarrow n|\cancel{k}}$ is the
hitting time from the $m$th site to the $n$th sitebefore hitting the $k$th site. 
\item 
Here, it is important and als convenient for us to let
\[
L_{n}=\phi_{(n-1)\rightarrow n|\cancel{(n-2)}}\cdot\phi_{n\rightarrow(n-1)|\cancel{(n+1)}}
\]
denote the moment generating function of the hitting time of the
loop between the (consecutive) $(n-1)$th site and the $n$th site.
\end{itemize}
\end{defn}

\begin{example}
We first recall the $1$-loop and $2$-loop cases, already studied
in \cite{JiuVignat}. 
\begin{figure}
\begin{centering}
\includegraphics[scale=0.5]{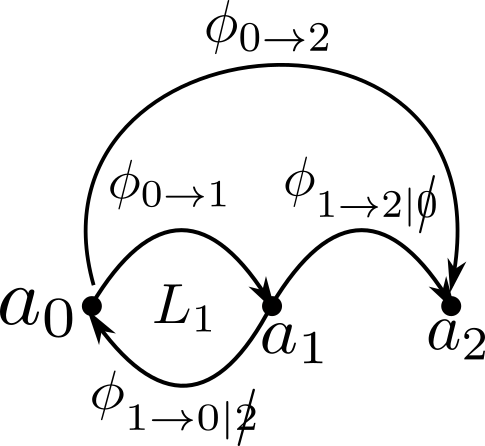}
\par\end{centering}
\caption{$1$-loop case\label{fig:1loop}}
\end{figure}
 The $1$-loop case can be view in Fig.~\ref{fig:1loop}, in which
we assume the initial site is $a_{0}$, namely there is no other site
to the left of $a_{0}$.  It is not hard to see the hitting time decomposition as follows:  
\[
t_{0\rightarrow2}=t_{0\rightarrow1}+\underset{k\text{ copies}}{\underbrace{t_{1\rightarrow0|\cancel{2}}+t_{0\rightarrow1}+\cdots+t_{1\rightarrow0|\cancel{2}}+t_{0\rightarrow1}}}+t_{1\rightarrow2|\cancel{0}},
\]
namely, there can be $k$ copies of $L_{1}$ in the moment generating
functions, for $k=0,1,2,\ldots$. When considering the moment generating
functions of both sides, independence turns the summation into products.
Therefore, we have 
\begin{equation}
\phi_{0\rightarrow2}=\phi_{0\rightarrow1}\phi_{1\rightarrow2|\cancel{0}}\sum_{k=0}^{\infty}\left(\phi_{1\rightarrow0|\cancel{2}}\phi_{0\rightarrow1}\right)^{k}=\frac{\phi_{0\rightarrow1}\phi_{1\rightarrow2|\cancel{0}}}{1-L_{1}}.\label{eq:1loop}
\end{equation}
(See also \cite[Eq.~(2.5)]{JiuVignat}.) In addition, the $2$-loop case is 
\begin{equation}
\phi_{0\rightarrow3}=\frac{\phi_{0\rightarrow1}\phi_{1\rightarrow2|\cancel{0}}\phi_{2\rightarrow3|\cancel{1}}}{1-(L_{1}+L_{2})},\label{eq:2loop}
\end{equation}
by a combinatorial enumeration \cite[Eq.~(2.6)]{JiuVignat}. 
\end{example}

We shall give the loop decomposition for $n$-loops in the following section.

\section{\label{sec:NLoop}General $n$-loop Decomposition}

 In this section, we shall give the expression of the general $n$-loop
formula (see Fig.~\ref{fig:Mloop}, the black paths), as the generalization
of (\ref{eq:1loop}) and (\ref{eq:2loop}). Again, we assume the walk
begins at the site $a_{0}$ and only walk to its right on consecutive
sites $a_{1}<a_{2}<\cdots<a_{m+1}$. Namely the walk will consider $a_0$ as its starting point and there are no 
sites to the left of $a_0$.

\begin{thm}\label{thm:LOOP}
\begin{eqnarray*}
& & \phi_{0\rightarrow (n+1)} =  \phi_{0\rightarrow1}\prod_{j=1}^{n}\phi_{j\rightarrow(j+1)|\cancel{j-1}}\, \left(
\sum_{k\geq 0} \sum_{**}\prod_{t=1}^k L_{i_t}\right)\\[0.2cm]
& & = \phi_{0\rightarrow1}\prod_{j=1}^{n}\phi_{j\rightarrow(j+1)|\cancel{j-1}}\, \sum_{k\geq 0} \left((L_1+L_2+\cdots + L_n)+\sum_{*'}' (-1)^{l+1} (L_{j_1}L_{j_2}\cdots L_{j_l}) \right)^k\\[0.2cm]
& & = \phi_{0\rightarrow1}\prod_{j=1}^{n}\phi_{j\rightarrow(j+1)|\cancel{j-1}}\, 
\frac{1}{1-(L_1+L_2+\cdots + L_n)+\sum\limits_{*'}'(-1)^{l} (L_{j_1}L_{j_2}\cdots L_{j_l})},
\end{eqnarray*}
where 
\[
**=\{(i_1,i_2,\cdots, i_k): 1 \leq i_t \leq n, \text{and } i_t = i_{t+1}, i_t =i_{t+1}+1, \text{or } i_t < i_{t+1}\};
\]
and
\[
*'=\{n\geq j_{1} > \cdots >j_{l}\ge 1,\;  l \geq 2, \; j_{m}-j_{m+1}\ge2\}.
\]
\end{thm}

\begin{rem*}
As one can tell, terms in $*'$ are loops without consecutive ones; and they are listed in a {\emph{descending}} order, for its combinatorial interpretation later in the proof. Before the proof of Thm.~\ref{thm:LOOP}, we would like to present several example for small number of loops. 
\begin{example}
\label{exa:2-5loops}The formulas for $n=2,3,4,5$, given by Theorem (\ref{thm:LOOP}) are listed as follows.

\begin{align*}
\phi_{0\rightarrow3} & =\frac{\phi_{0\rightarrow1}\phi_{1\rightarrow2|\cancel{0}}\phi_{2\rightarrow3|\cancel{1}}}{1-(L_{1}+L_{2})},\\
\phi_{0\rightarrow4} & =\frac{\phi_{0\rightarrow1}\phi_{1\rightarrow2|\cancel{0}}\phi_{2\rightarrow3|\cancel{1}}\phi_{3\rightarrow4|\cancel{2}}}{1-(L_{1}+L_{2}+L_{3}-L_{3}L_{1})},\\
\phi_{0\rightarrow5} & =\frac{\phi_{0\rightarrow1}\phi_{1\rightarrow2|\cancel{0}}\phi_{2\rightarrow3|\cancel{1}}\phi_{3\rightarrow4|\cancel{2}}\phi_{4\rightarrow5|\cancel{3}}}{1-(L_{1}+L_{2}+L_{3}+L_{4}-L_{4}L_{2}-L_{4}L_{1}-L_{3}L_{1})},\allowdisplaybreaks\\
\phi_{0\rightarrow6} & =\frac{\phi_{0\rightarrow1}\phi_{1\rightarrow2|\cancel{0}}\phi_{2\rightarrow3|\cancel{1}}\phi_{3\rightarrow4|\cancel{2}}\phi_{4\rightarrow5|\cancel{3}}\phi_{5\rightarrow6|\cancel{4}}}{1-(L_{1}+\cdots+L_{5}-L_{5}L_{3}-L_{5}L_{2}-L_{5}L_{1}-L_{4}L_{2}-L_{4}L_{1}-L_{3}L_{1}+L_{5}L_{3}L_{1})}.
\end{align*}
Apparently, if there are only two loops, $L_1$ and $L_2$, then $*$, as well as $*'$ is empty. So we can recover the case of two loops, i.e., \eqref{eq:2loop}. 
\end{example}

\begin{proof}[Proof of Thm.~\ref{thm:LOOP}, by inclusion-exclusion principle]

The proof will  give  a combinatorial interpretation to the terms appearing in  Theorem \ref{thm:LOOP}. 
\vspace{0.2cm}

Let $n\geq 1$. An arbitrary decomposition of $t_{0\rightarrow (n+1)}$ as a sum of hitting times between  consecutive sites, which generalizes the one given  in Example 2.2, can be written as

\[
(A)\qquad t_{0\rightarrow (n+1)}= \sum_{j=0}^{2k+n+1} t_{i_j \rightarrow i_{j+1}|\cancel{i_j^*}}
\]
where
\vspace{0.2cm}

\begin{itemize}
\item[(i)] $i_0=0,\,  i_1 =1,\,  i_{2k+n+1}=n+1$;\\[0.2cm]

\item[(ii)]  if $\; 0 \leq j < 2k+n$, $\quad 0 \leq i_j , i_{j+1} \leq n$, $\quad $ and $\quad \; |i_j-i_{j+1}|=1$;\\[0.2cm]

\item[(iii)] if $\; i_j \not = 0, $ $\; |i_j^* - i_j|=1 \; $ and $\; |i_j^* - i_{j+1}|=2$;\\[0.2cm]

\item[(iv)] if $\; i_j = 0,$  $\; t_{i_j \rightarrow i_{j+1}|\cancel{i_j^*}} =  t_{i_j \rightarrow i_{j+1}} = t_{0\rightarrow 1}.$
\end{itemize}
 \vspace{0.3cm}

 It is not hard to see that as we move through the sum in $(A)$, in the direction of increasing $j$, every time  we encounter a stopping time which corresponds to a step to the left, i.e., from a site $t_j$ to site $t_{j+1}=t_j-1$, a loop is formed. These loops can only occur between consecutive sites, and the order in which they appear is important: different walks with the same collection of loops will have different orderings. The ordering of the loops has an additional constraint, given by the assumption the walk only moves between consecutive sites.  That is, if $t_{j+1} < t_j$, then $t_{j+1}=t_j -1$. No such constraint is needed if $t_{j+1} > t_j$.  Hence every decomposition of $t_{0\rightarrow (n+1)}$  in (A) can be uniquely written as   
\[
(B) \qquad t_{0\rightarrow (n+1)}=   
t_{0\rightarrow 1} +  \sum_{j=1}^{n}t_{j\rightarrow(j+1)|\cancel{j-1}} + \sum_{t=1}^{k} l_{i_t}
\]

where

\begin{itemize}

\item[(a)] the order in which we sum the  $l_{i_t}$'s is important;\\[0.2cm]

\item[(b)] $1 \leq i_t \leq n$, and  $i_t = i_{t+1}$, or  $\; i_t =i_{t+1} +1$, or $i_t < i_{t+1}$.

\end{itemize}
\vspace{0.2cm}

\noindent
Clearly this correspondence can be reversed, and (B) can  be used to express  the moment generating function of $t_{0\rightarrow (n+1)}$ in terms of the moment generating functions of loops. That is, 
 
\begin{equation} \label{new form 0}
\phi_{0\rightarrow n+1}=\phi_{0\rightarrow1}\prod_{j=1}^{n}\phi_{j\rightarrow(j+1)|\cancel{j-1}}\, \left(
\sum_{k\geq 0} \sum_{**}\prod_{t=1}^k L_{i_t}\right),
\end{equation}
where 
\[ **=\{(i_1,i_2,\cdots, i_k):\; \mbox{the $ i_t$'s  satisfy $(b).$}\}
\]
\vspace{0.3cm}

Let $k$ be arbitrary but fixed.  We are going to apply the method of inclusion exclusion to recover  $ \sum\limits_{**}\prod\limits_{t=1}^k L_{i_t}$ from $(L_1+L_2+\cdots + L_n)^k$. 

Note that, if $n=1, 2$, (b) necessarily holds,  and  from (\ref{new form 0}) we immediately obtain

\[
\phi_{0\rightarrow 2} = \phi_{0\rightarrow1}\phi_{1\rightarrow2|\cancel{0}} \sum_{k=0}^{\infty} L_1^k = \frac{\phi_{0\rightarrow1}\phi_{1\rightarrow2|\cancel{0}}}{1-L_1}, 
\]
and
\[
\phi_{0\rightarrow 3}  =  \phi_{0\rightarrow1}\phi_{1\rightarrow2|\cancel{0}}\phi_{2\rightarrow3|\cancel{1}}
 \sum_{k=0}^{\infty} (L_1+L_2)^k = \frac{\phi_{0\rightarrow1}\phi_{1\rightarrow2|\cancel{0}}\phi_{2\rightarrow3|\cancel{1}}
}{1-(L_1+L_2)},
\]
which coincide with (\ref{eq:1loop}) and (\ref{eq:2loop}), respectively (see also Example (\ref{exa:2-5loops})).

Let $n \geq 3$. Before applying the method of inclusion exclusion it is convenient to  introduce an additional notation. Since property (b) can be expressed in terms of the order in which multiplication of the $L_j$'s is performed,  we write
$$(L_1+L_2+\cdots +L_n)^k = \prod_{t=1}^k \, (L_1+L_2+\cdots + L_n)_{(t)}$$
and use these additional subscripts  to label the  forbidden products of pair of moment generating functions of loops (m.g.f. of loops) in the expansion of $(L_1+L_2+\cdots + L_n)^k$. We write   $(L_iL_j)_{(s)}$ when $i-j>1$, $L_i$ comes from the factor $(L_1+L_2+\cdots L_n)_{(s)}$, and $L_j$ comes from  $(L_1+L_2+\cdots + L_n)_{(s+1)}$. Moreover  when we write $ (L_iL_j)_{(s)}\, (L_r L_u)_{(v)},$   we assume that
$$  v \geq s+1,\quad \mbox{ and } v=s+1 \quad \mbox{if and only if } \quad L_j=L_r.$$
In this latter case,
\begin{equation} \label{convention 1}
(L_iL_j)_{(s)}\, (L_j L_u)_{(s+1)} \quad \mbox{ reduces to } \quad  (L_i L_j L_u)_{(s)}.
\end{equation}
 (\ref{convention 1}) naturally extends to products of several forbidden pairs of m.g.f. of loops, thus producing forbidden tuples $(L_{i_1} L_{i_2}\cdots L_{i_t})$, where $i_j \geq i_{j+1} + 2$, $1\leq j \leq t-1$.
\vspace{0.2cm}

The method of inclusion exclusion now gives that $\sum\limits_{k\geq 0} \sum\limits_{**}\prod\limits_{t=1}^k L_{i_t} $ can be written as

\begin{equation} \label{new form 1}
\sum_{k\geq 0} \left[ (L_1 +\cdots + L_n)^k + 
\sum_{1 \leq l \leq k-1} (-1)^{l} \sum_{(l)} 
 \big(\prod_{j=1}^l (L_{i_j}L_{t_j})_{s_j} (L_1 + \cdots + L_n)^{n-\#_l} \big)\right],
 \end{equation} 
where $\#_l$ denotes the number of distinct $L_r$ in  $\prod\limits_{j=1}^l (L_{i_j}L_{t_j})_{s_j} $, and $\sum\limits_{(l)}$ runs over of all possible  $\{(i_j,t_j)_{s_j}, 1\leq j \leq l, s_j < s_{j+1}, \, i_j > t_{j}+1\}.$ Note that this set could be empty. 

Next we rewrite (\ref{new form 1})   in a  very simple form.  First,  whenever possible we apply (\ref{convention 1}) and its extensions, then we drop the subscripts (now they do not provide any additional information),  but we keep parentheses around each forbidden quantity.  Finally, we combine like terms. We view each forbidden tuple, i.e, $ (L_{i_1}L_{i_2}), (L_{i_3}L_{i_4}L_{i_5}), \cdots $ as distinct variables, and,  for every $k$, we collect all monomials  of degree $k$ in these variables. We claim that this procedure reduces  (\ref{new form 1}) to
\begin{equation} \label{new form 2}
\sum_{k\geq 0} \left((L_1+L_2+\cdots + L_n)+\sum_{*'} (-1)^{l+1} (L_{j_1}L_{j_2}\cdots L_{j_l}) \right)^k,
 \end{equation}
 where $*'=\{n\geq j_{1} > \cdots >j_{l}\ge 1,\;  l \geq 2, \; j_{m}-j_{m+1}\ge2\}$.

 We show the equivalence of (\ref{new form 1}) and (\ref{new form 2}) by showing that, for every $k \geq 0$, there is a one to one correspondence between the terms of degree $k$ of these two quantities. Let $k=m$ be arbitrary but fixed. An arbitrary term  in the expansion of 
 $$\left((L_1+L_2+\cdots + L_n)+\sum_{*'} (-1)^{l+1} (L_{j_1}L_{j_2}\cdots L_{j_l}) \right)^m$$
 in (\ref{new form 2}) will have  the form
   \[
 (-1)^s{m \choose t_1,\cdots,t_{r+1}} 
  \left(\prod_{i=1}^{s_1}L_ {j^{(1)}_{i}}\right)^{t_1}\,  \left(\prod_{i=1}^{s_2}L_ {j^{(2)}_{i}}\right)^{t_2}\cdots  \left(\prod_{i=1}^{s_r}L_ {j^{(r)}_{i}}\right)^{t_r}\, (L_1+\cdots + L_n)^{t_{r+1}}
 \]
 for arbitrary positive integers $t_1, t_2, \cdots ,t_{r+1}$, $t_1 + t_2 +\cdots + t_{r+1}=m$,   arbitrary forbidden tuples $(j^{(1)}_{1},\cdots , j^{(1)}_{s_1}), \cdots , (j^{(r)}_{1},\cdots , j^{(r)}_{s_r})$, and  $s=(s_1+1)t_1 + \cdots + (s_r+1) t_r$. 
 In  (\ref{new form 1}) the terms
 \begin{equation} \label{shift}
   (-1)^s\, \left(\prod_{i=1}^{s_1}L_ {j^{(1)}_{i}}\right)^{t_1}\,  \left(\prod_{i=1}^{s_2}L_ {j^{(2)}_{i}}\right)^{t_2}\cdots  \left(\prod_{i=1}^{s_r}L_ {j^{(r)}_{i}}\right)^{t_r}\, (L_1+\cdots + L_n)^{t_{r+1}}
 \end{equation}
 are introduced when  the method of inclusion exclusion is applied
to remove the forbidden pairs of m.g.f. of loops from 
 \[\left(L_1+\cdots + L_n\right)^{t_1s_1+t_2s_2+\cdots + t_r s_r + t_{r+1}}.
 \]
 The number of these terms  can  be obtained by viewing each pair of parentheses $(\cdots)$ in (\ref{shift})  as  distinct objects,  $(\cdots)^{t_i}$ implies $(\cdots)$ is repeated $t_i$ times,  and counting the number of ways they can be arranged on a line, order is important. It is easy to see that this number is 
 
 \[\frac{(t_1+t_2+\cdots + t_r + t_{r+1})!}{t_1!t_2!\cdots t_r!\, t_{r+1}!} =  {m \choose t_1,t_2,\cdots ,t_r,t_{r+1}}.
 \]
 This shows that all the terms of degree $m$ in (\ref{new form 2}) can be found in (\ref{new form 1}). Moreover it is easy to see that these are the only terms of degree $m$  in (\ref{new form 1}), thus proving the equivalence between (\ref{new form 1}) and (\ref{new form 2}). 
 
 Hence,
 \begin{eqnarray*}
 & & \phi_{0\rightarrow (n+1)} =  \phi_{0\rightarrow1}\prod_{j=1}^{n}\phi_{j\rightarrow(j+1)|\cancel{j-1}}\, \left(
 \sum_{k\geq 0} \sum_{**}\prod_{t=1}^k L_{i_t}\right)\\[0.2cm]
& & = \phi_{0\rightarrow1}\prod_{j=1}^{n}\phi_{j\rightarrow(j+1)|\cancel{j-1}}\, \sum_{k\geq 0} \left((L_1+L_2+\cdots + L_n)+\sum_{*'}' (-1)^{l+1} (L_{j_1}L_{j_2}\cdots L_{j_l}) \right)^k\\[0.2cm]
& & = \phi_{0\rightarrow1}\prod_{j=1}^{n}\phi_{j\rightarrow(j+1)|\cancel{j-1}}\, 
\frac{1}{1-(L_1+L_2+\cdots + L_n)+\sum_{*'}' (-1)^{l} (L_{j_1}L_{j_2}\cdots L_{j_l})}.
 \end{eqnarray*}
The proof is now complete.
\end{proof}

\newpage

Next we provide an alternative proof of Theorem (\ref{thm:LOOP}) by using induction on $n$. To simplify the notation in the proof, we shall use
\begin{equation}
\sum\limits _{(k,l,n)}=\sum\limits _{*}(-1)^{l+1}L_{j_{1}}\cdots L_{j_{l}},\label{eq:SummationNotation}
\end{equation}
{\large{}w}here $*=\{k=j_{1}<\cdots<j_{l}\le n,1\le l\le n-k+1,j_{m+1}-j_{m}\ge2\}$, in which we reordered 
the loops in the ascending order.\\
\begin{rem*}
Notice that we have reversed the order of subscript in the newly defined product notation. Indeed, $\sum\limits_{*}$ and $\sum\limits_{*'}$ are mathematically equivalent. Although we believe $\sum\limits_{*'}$ does a better job in conveying the combinatorics idea behind the loop identity, it is for the simplicity of expression that we choose to use the reversed order in the following proof.
\end{rem*}
\noindent It is easy to see that 
\begin{equation}
\sum\limits _{(k,l,n)}=L_{k}-L_{k}\sum\limits _{j=k+2}^{n}\sum\limits _{(j,l,n)}=L_{k}\left(1-\sum\limits _{j=k+2}^{n}\sum\limits _{(j,l,n)}\right).\label{eq:LoopExpansion}
\end{equation}
Hence, we can further rewrite Thm.~\ref{thm:LOOP} as 
\begin{equation}
\phi_{0\rightarrow n+1}=\phi_{0\rightarrow1}\prod_{j=1}^{n}\phi_{j\rightarrow(j+1)|\cancel{j-1}}\frac{1}{1-\sum\limits _{k=1}^{n}\sum\limits _{(k,l,n)}}.\label{eq:Main}
\end{equation}
\end{rem*}

\begin{figure}
\includegraphics[scale=0.3]{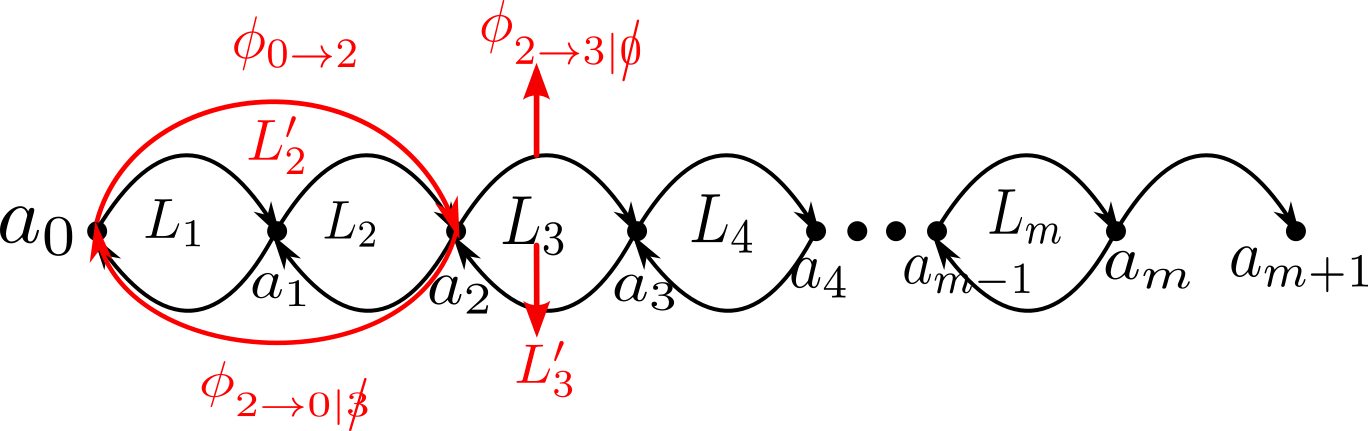}
\caption{\label{fig:Mloop}$m$-loop}
\end{figure}

\begin{proof}[Proof of Thm.~\ref{thm:LOOP}, by induction]
For the case $n=1$, \eqref{eq:Main}
reduces to 
\begin{equation}
\phi_{0\rightarrow2}=\phi_{0\rightarrow1}\phi_{1\rightarrow2|\cancel{0}}\frac{1}{1-L_{1}},\label{eq:n2}
\end{equation}
the same as \eqref{eq:1loop}. Suppose \eqref{eq:Main} holds for
$n=m-1$. Then, we need to show 
\begin{equation}
\phi_{0\rightarrow(m+1)}=\phi_{0\rightarrow1}\prod_{j=1}^{m}\phi_{j\rightarrow(j+1)|\cancel{j-1}}\frac{1}{1-\sum\limits _{k=1}^{m}\sum\limits _{(k,l,m)}}.\label{eq:Inductive}
\end{equation}
We combine the first two loops together, which reduces to $m-1$ new
loops, labeled as $L_{2}',L_{3}',L_{4}',\ldots$. See Fig.~\ref{fig:Mloop}.
Note that $\phi_{0\rightarrow2}$ is given by (\ref{eq:n2}); and
similarly 
\[
\phi_{2\rightarrow0|\cancel{3}}=\frac{\phi_{1\rightarrow2|\cancel{0}}\phi_{2\rightarrow1|\cancel{3}}}{1-\phi_{1\rightarrow2|\cancel{0}}\phi_{2\rightarrow1|\cancel{3}}}.
\]
Hence, 
\[
L'_{2}=\phi_{0\rightarrow2}\phi_{2\rightarrow0|\cancel{3}}=\frac{L_{1}L_{2}}{(1-L_{1})(1-L_{2})}.
\]
In addition, 
\begin{align*}
L'_{3} & =\phi_{2\rightarrow3|\cancel{0}}\phi_{3\rightarrow2|\cancel{4}}\\
 & =\phi_{2\rightarrow3|\cancel{1}}\sum_{k=0}^{\infty}(\phi_{2\rightarrow1|\cancel{3}}\phi_{1\rightarrow2|\cancel{0}})^{k}\phi_{3\rightarrow2|\cancel{4}}\\
 & =(\phi_{2\rightarrow3|\cancel{1}}\phi_{3\rightarrow2|\cancel{4}})\frac{1}{1-\phi_{1\rightarrow2|\cancel{0}}\phi_{2\rightarrow1|\cancel{3}}}\\
 & =\frac{L_{3}}{1-L_{2}},
\end{align*}
and $L'_{k}=L_{k}\text{, for all }4\le k\le m$. To further simplify
expressions, we need a new summation symbol: 
\[
\sum'\limits _{(k,l,n)}=\sum\limits _{*}(-1)^{l+1}L'_{j_{1}}\cdots L'_{j_{l}},
\]
where $*=\{k=j_{1}<\cdots<j_{l}\le n,1\le l\le n-k+1,j_{m+1}-j_{m}\ge2\}$,
which is exactly the same as (\ref{eq:SummationNotation}), with all
$L$'s replaced by $L'$. 

Now apply (\ref{eq:Main}) for sites $0,2,3,\ldots,m+1$, (i.e., with
$m-1$ loops,) to get 
\begin{align*}
\phi_{0\rightarrow(m+1)} & =\phi_{0\rightarrow2}\phi_{2\rightarrow3|\cancel{0}}\prod_{j=3}^{m}\phi_{j\rightarrow(j+1)|\cancel{j-1}}\frac{1}{1-\sum\limits _{k=2}^{m}\sum'\limits _{(k,l,m)}}\\
 & =\phi_{0\rightarrow1}\phi_{1\rightarrow2|\cancel{0}}\frac{1}{1-L_{1}}\phi_{2\rightarrow3|\cancel{1}}\frac{1}{1-L_{2}}\prod_{j=3}^{m}\phi_{j\rightarrow(j+1)|\cancel{j-1}}\frac{1}{1-\sum\limits _{k=2}^{m}\sum'\limits _{(k,l,m)}}\\
 & =\phi_{0\rightarrow1}\prod_{j=1}^{m}\phi_{j\rightarrow(j+1)|\cancel{j-1}}\frac{1}{(1-L_{1})(1-L_{2})\left(1-\sum\limits _{k=2}^{m}\sum'\limits _{(k,l,m)}\right)}.
\end{align*}
Therefore, (\ref{eq:Inductive}) is equivalent to 
\begin{equation}
(1-L_{1})(1-L_{2})\left(1-\sum\limits _{k=2}^{m}\sum'\limits _{(k,l,m)}\right)=1-\sum\limits _{k=1}^{m}\sum\limits _{(k,l,m)}.\label{eq:ProofExpansion}
\end{equation}
By applying (\ref{eq:LoopExpansion}), we have the left-hand side
\begin{align*}
 & \quad(1-L_{1})(1-L_{2})\left(1-\sum'\limits _{(2,l,m)}-\sum'\limits _{(3,l,m)}-\sum\limits _{k=4}^{m}\sum'\limits _{(k,l,m)}\right)\allowdisplaybreaks\\
 & =1-L_{1}-L_{2}+L_{1}L_{2}-(1-L_{1})(1-L_{2})\left[L'_{2}\left(1-\sum\limits _{k=4}^{m}\sum'\limits _{(k,l,m)}\right)\right.\\
 & \quad\left.+L'_{3}\left(1-\sum\limits _{k=5}^{m}\sum'\limits _{(k,l,m)}\right)\right]-\sum\limits _{k=4}^{m}\sum'\limits _{(k,l,m)}+L_{1}\sum\limits _{k=4}^{m}\sum'\limits _{(k,l,m)}+L_{2}\sum\limits _{k=4}^{m}\sum'\limits _{(k,l,m)}\\
 & \quad-L_{1}L_{2}\sum\limits _{k=4}^{m}\sum'\limits _{(k,l,m)}\\
 & =1-L_{1}-L_{2}+\underline{L_{1}L_{2}-L_{1}L_{2}\left(1-\sum\limits _{k=4}^{m}\sum'\limits _{(k,l,m)}\right)}-L_{3}\left(1-\sum\limits _{k=5}^{m}\sum'\limits _{(k,l,m)}\right)\\
 & \quad+L_{1}L_{3}\left(1-\sum\limits _{k=5}^{m}\sum'\limits _{(k,l,m)}\right)-\sum\limits _{k=4}^{m}\sum'\limits _{(k,l,m)}+L_{1}\sum\limits _{k=4}^{m}\sum'\limits _{(k,l,m)}+L_{2}\sum\limits _{k=4}^{m}\sum'\limits _{(k,l,m)}\\
 & \quad\underline{-L_{1}L_{2}\sum\limits _{k=4}^{m}\sum'\limits _{(k,l,m)}}.
\end{align*}
Note that the two term underlined cancel; and since $L'_{k}=L_{k}$
for $k\geq4$, all the $\sum\limits'$ terms above are actually $\sum$.
Therefore, the left-hand side of (\ref{eq:ProofExpansion}) is 
\begin{align*}
 & =1-L_{1}\left[1-L_{3}\left(1-\sum\limits _{k=5}^{m}\sum\limits _{(k,l,m)}\right)-\sum\limits _{k=4}^{m}\sum\limits _{(k,l,m)}\right]-L_{2}\left(1-\sum\limits _{k=4}^{m}\sum\limits _{(k,l,m)}\right)\allowdisplaybreaks\\
 & \quad-L_{3}\left(1-\sum\limits _{k=5}^{m}\sum\limits _{(k,l,m)}\right)-\sum\limits _{k=4}^{m}\sum\limits _{(k,l,m)}\\
 & =1-\sum\limits _{k=1}^{m}\sum\limits _{(k,l,m)},
\end{align*}
which is the right-hand side of (\ref{eq:ProofExpansion}). 
\end{proof}

\vspace{0.2cm}

\section{Preliminaries: umbral random symbols}\label{sec:umbral}

Again, one can find similar summary of this section in \cite{JiuVignat}; but we are restating them here, for self-containedness. We will let $\mathcal{B}$, $\mathcal{E}$, and $\mathcal{U}$ be
the \emph{Bernoulli}, \emph{Euler}, and \emph{uniform (umbral) symbols},
respectively. They are defined as follows. 

\subsection{Bernoulli $\mathcal{B}$}

The Bernoulli symbol $\mathcal{B}$ satisfies the evaluation rule
\begin{equation}
(x+\mathcal{B})^{n}=B_{n}(x).\label{eq:BernoulliBEval}
\end{equation}
In fact, $\mathcal{B}$ can be viewed as a random variable \cite[Thm.~2.3]{Zagier},
i.e., $\mathcal{B}=iL_{B}-1/2$, for $i^{2}=-1$ and $L_{B}$ is the
random variable on $\mathbb{R}$, with density $p_{B}(t)=\pi\sech^{2}(\pi t)/2$.
Hence, the evaluation rule is equivalent to the expectation operator.
Moreover, for any suitable function $f$ (, i.e., one making integrals absolutely convergent,)
\[
f(x+\mathcal{B})=\mathbb{E}\left[f\left(x+iL_{B}-\frac{1}{2}\right)\right]=\frac{\pi}{2}\int_{\mathbb{R}}f\left(x+it-\frac{1}{2}\right)\sech^{2}(\pi t)dt.
\]
In particular, $f(x)=x^{n}$ yields (\ref{eq:BernoulliBEval}). In addition,
we have 
\[
B_{n}^{(p)}(x)=\left(x+\mathcal{B}^{(p)}\right)^{n}=\left(x+\mathcal{B}_{1}+\cdots+\mathcal{B}_{p}\right)^{n}
\]
for a set of $p$ {\emph{independent}} umbral symbols (or random variables)
$\left(\mathcal{B}_{i}\right)_{i=1}^{p}$, satisfying:
\begin{itemize}
\item if $i\neq j$, so that $\mathcal{B}_{i}$ and $\mathcal{B}_{j}$ are
independent, then we evaluate 
\[
\mathcal{B}_{i}^{n}\mathcal{B}_{j}^{m}=B_{n}B_{m};
\]
\item and if $i=j$, then 
\[
\mathcal{B}_{i}^{n}\mathcal{B}_{j}^{m}=\mathcal{B}_{i}^{n+m}=B_{n+m}.
\]
\end{itemize}
Now, with (\ref{eq:GF}), we deduce that 
\begin{equation}
e^{\mathcal{B}t}=\frac{t}{e^{t}-1},\quad e^{t(2\mathcal{B}+1)}=\frac{t}{\sinh t},\quad\text{and}\quad e^{t(2\mathcal{B}^{(p)}+p)}=\left(\frac{t}{\sinh t}\right)^{p}.\label{eq:Bp}
\end{equation}

\subsection{Euler $\mathcal{E}$. }

The Euler symbol $\mathcal{E}$ can be similarly defined via the random
variable interpretation that $\mathcal{E}=iL_{E}-1/2$, where $L_{E}$'s
density is given by $p_{E}(t)=\sech(\pi t)$. Then, $(\mathcal{E}+x)^n=E_n(x)$; and in particular,
for sum of independent symbols $\mathcal{E}^{(p)}=\mathcal{E}_{1}+\cdots+\mathcal{E}_{p}$,
\[
E_{n}^{(p)}(x)=\left(x+\mathcal{E}^{(p)}\right)^{n}.
\]
Therefore, (\ref{eq:GF}) yields
\begin{equation}
e^{t\mathcal{E}}=\frac{2}{e^{t}+1},\quad e^{t(2\mathcal{E}+1)}=\sech t,\quad\text{and}\quad e^{t(2\mathcal{E}^{(p)}+p)}=\sech^{p}t.\label{eq:Ep}
\end{equation}
Moreover, from the generating functions, \eqref{eq:Bp} and \eqref{eq:Ep},
\[
e^{2\mathcal{B}t}=\frac{2t}{e^{2t}-1}=\frac{t}{e^{t}-1}\cdot\frac{2}{e^{t}+1}=e^{t(\mathcal{B}+\mathcal{E})},
\]
we can have $2\mathcal{B}=\mathcal{B}+\mathcal{E}$, namely for any
suitable function $f$, $f(x+2\mathcal{B})=f(x+\mathcal{B}+\mathcal{E})$. 

\subsection{Uniform $\mathcal{U}$}

The uniform symbol $\mathcal{U}$ is the uniform random variable on
$[0,1]$, i.e., $\mathcal{U}\sim U[0,1]$, so that the evaluation
is 
\begin{equation}
\mathcal{U}^{n}=\int_{0}^{1}t^{n}dt=\frac{1}{n+1}.\label{eq:RVU}
\end{equation}
Easily, we have 
\begin{equation}
e^{t\mathcal{U}}=\sum_{n=0}^{\infty}\frac{t^{n}}{(n+1)!}=\frac{e^{t}-1}{t},\quad e^{t(2\mathcal{U}-1)}=\frac{\sinh t}{t},\quad\text{and}\quad e^{t(2\mathcal{U}^{(p)}-p)}=\left(\frac{\sinh t}{t}\right)^{p},\label{eq:Up}
\end{equation}
for the sum of independent symbols~$\mathcal{U}^{(p)}=\mathcal{U}_{1}+\cdots+\mathcal{U}_{p}$.
An important link between $\mathcal{B}$ and $\mathcal{U}$ is the
cancellation rule. Note that 
\[
e^{t(\mathcal{U}+\mathcal{B})}=e^{t\mathcal{U}}e^{t\mathcal{B}}=\frac{e^{t}-1}{t}\cdot\frac{t}{e^{t}-1}=1.
\]
So for a suitable function $f$, 
\[
f(x+\mathcal{B}+\mathcal{U})=f(x).
\]

In what follows, we will use independent copies of the three symbols.
In order to distinguish them, we shall denote independent uniform
symbols by $\mathcal{U},\mathcal{U}',\ldots$ and $\mathcal{U}^{(p)},\mathcal{U}'^{(p)},\ldots$;
and similarly for the other two symbols.

\section{Identities of Bernoulli and Euler polynomials}\label{sec:1loop}

Now, with the loop decomposition (\ref{eq:Main}) and evaluation of
symbols (\ref{eq:Bp}), (\ref{eq:Ep}), and (\ref{eq:Up}), we can
derive certain identities. Note that, even in the case of two loops,
it does not seem possible to further simplify the expressions completely
in terms of Bernoulli and Euler polynomials, but in terms of the three
symbols; see e.g., \cite[Thms.~3.4 and 4.2 ]{JiuVignat}. We would
like to set all the sites \emph{equally distributed}, namely $a_{j}=j$,
for $j=0,1,2,\ldots,n$, throughout this section. 

\subsection{$1$-dim reflected Brownian motion on $\mathbb{R}_{+}$}

In this case, for three consecutive sites $a<b<c$, the generating
functions of the corresponding hitting times can be found in \cite[p.~198 and p.~355]{Formulas}
with variable $w$: 
\begin{align*}
\phi_{a\rightarrow b} & =\frac{\cosh(aw)}{\cosh(bw)},\\
\phi_{b\rightarrow a|\cancel{c}} & =\frac{\sinh\left((c-b)w\right)}{\sinh\left((c-a)w\right)},\\
\phi_{b\rightarrow c|\cancel{a}} & =\frac{\sinh\left((b-a)w\right)}{\sinh\left((c-a)w\right)}.
\end{align*}
In this case, we begin with $a_{0}=0$ as the initial site and then
apply the formulas above to have, for $n\geq1$. 
\begin{equation}
\phi_{0\rightarrow n}=\frac{1}{\cosh(nw)}\quad\text{and}\quad\phi_{n\rightarrow n+1|\cancel{n-1}}=\phi_{n\rightarrow n-1|\cancel{n+1}}=\frac{1}{2\cosh(w)}.\label{eq:1dim}
\end{equation}

Before we state and prove the general formula, we first compute an
example of $3$-loop case, which is not included in \cite{JiuVignat}. 
\begin{example}
As stated in Ex.~\ref{exa:2-5loops}, 
\[
\phi_{0\rightarrow4}=\frac{\phi_{0\rightarrow1}\phi_{1\rightarrow2|\cancel{0}}\phi_{2\rightarrow3|\cancel{1}}\phi_{3\rightarrow4|\cancel{2}}}{1-(L_{1}+L_{2}+L_{3}-L_{1}L_{3})}.
\]
Now, apply (\ref{eq:1dim}) to have 
\begin{align*}
\frac{1}{\cosh(4w)} & =\frac{1}{8\cosh^{4}w}\sum_{k=0}^{\infty}\left(\frac{\sinh w}{\sinh(2w)\cosh w}+\frac{2\sinh^{2}w}{\sinh^{2}(2w)}\right.\\
 & \quad\left.-\frac{\sinh w}{\sinh(2w)\cosh w}\cdot\frac{\sinh^{2}w}{\sinh^{2}(2w)}\right)^{k}.
\end{align*}
\begin{itemize}
\item The left-hand side is simply $\sech(4w)=\exp\left\{ 4w(2\mathcal{E}+1)\right\} .$
\item For the right-hand side, we first simplify that
\begin{align*}
\frac{\sinh w}{\sinh(2w)\cosh w}+\frac{2\sinh^{2}w}{\sinh^{2}(2w)} & =\frac{1}{\cosh^{2}w},\\
\frac{\sinh w}{\sinh(2w)\cosh w}\cdot\frac{\sinh^{2}w}{\sinh^{2}(2w)} & =\frac{1}{8\cosh^{4}w}.
\end{align*}
Hence, we have
\begin{align*}
 & \quad\frac{1}{8\cosh^{4}w}\sum_{k=0}^{\infty}\left(\frac{1}{\cosh^{2}w}-\frac{1}{8\cosh^{4}w}\right)^{k}\\
 & =\frac{1}{8\cosh^{4}w}\sum_{k=0}^{\infty}\sum_{\ell=0}^{k}(-1)^{\ell}\binom{k}{\ell}\left(\frac{1}{\cosh^{2}w}\right)^{k-\ell}\left(\frac{1}{8\cosh^{4}w}\right)^{\ell}\\
 & =\sum_{k=0}^{\infty}\sum_{\ell=0}^{k}(-1)^{\ell}\binom{k}{\ell}\frac{1}{8^{\ell+1}}\cdot\left(\frac{1}{\cosh w}\right)^{2k+2\ell+4}\\
 & =\sum_{k=0}^{\infty}\sum_{\ell=0}^{k}(-1)^{\ell}\binom{k}{\ell}\frac{1}{8^{\ell+1}}\exp\left\{ w(2\mathcal{E}^{(2k+2\ell+4)}+2k+2\ell+4)\right\} .
\end{align*}
\end{itemize}
Namely, 
\[
\exp\left\{ 8\mathcal{E}w+4w\right\} =\sum_{k=0}^{\infty}\sum_{\ell=0}^{k}(-1)^{\ell}\binom{k}{\ell}\frac{1}{8^{\ell+1}}\exp\left\{ w(2\mathcal{E}^{(2k+2\ell+4)}+2k+2\ell+4)\right\} ,
\]
i.e., 
\[
\exp\left\{ 8\mathcal{E}w\right\} =\sum_{k=0}^{\infty}\sum_{\ell=0}^{k}(-1)^{\ell}\binom{k}{\ell}\frac{1}{8^{\ell+1}}\exp\left\{ w(2\mathcal{E}^{(2k+2\ell+4)}+2k+2\ell)\right\} .
\]
Multiplying both sides by $\exp\{xw\}$ and comparing the coefficients
of $w^{n}$ , we see
\begin{itemize}
\item the left-hand side yields
\[
\exp\left\{ (8\mathcal{E}+x)w\right\} \Rightarrow\left(8\mathcal{E}+x\right)^{n}=8^{n}E_{n}\left(\frac{x}{8}\right);
\]
\item while the right-hand side gives 
\begin{align*}
 & \quad\sum_{k=0}^{\infty}\sum_{\ell=0}^{k}(-1)^{\ell}\binom{k}{\ell}\frac{1}{8^{\ell+1}}(2\mathcal{E}^{(2k+2\ell+4)}+2k+2\ell+x)^{n}\\
 & =\sum_{k=0}^{\infty}\sum_{\ell=0}^{k}(-1)^{\ell}\binom{k}{\ell}\frac{2^{n}}{8^{\ell+1}}E_{n}^{(2k+2\ell+4)}\left(\frac{x}{2}+k+\ell\right).
\end{align*}
\end{itemize}
Therefore, we have 
\[
E_{n}(x)=\frac{1}{4^{n}}\sum_{k=0}^{\infty}\sum_{\ell=0}^{k}(-1)^{\ell}\binom{k}{\ell}\frac{2^{n}}{8^{\ell+1}}E_{n}^{(2k+2\ell)}\left(4x+k+\ell\right).
\]
\end{example}

\begin{rem*}
It is easy to see that $\phi_{0\rightarrow1}$ is twice of other paths:
$\phi_{n\rightarrow n+1|\cancel{n-1}}=\phi_{n\rightarrow n-1|\cancel{n+1}}$,
which also causes the difference between $L_{1}$ and $L_{j}$'s $j=2,3,\ldots$
The following combinatorial enumeration is the key to find the general
formulas.
\end{rem*}
\begin{defn}
Suppose $S=\{a_{1},a_{2},\dots,a_{n}\}$ are a set of a sequence of
$n$ mathematical objects. 
\begin{enumerate}
\item If the subindices of $a_{j_{1}}a_{j_{2}}\cdots a_{j_{m}}$ satisfy
$1\le j_{1}\le\cdots\le j_{m}\le n,j_{k}-j_{k-1}\ge2$, we call it
a \emph{nonadjacent product of order $n$, length $m$ with initial
state $j_{1}$}. 
\item We define $N(\ell,n)$ as \emph{the number of all nonadjacent products
of order $n$ and length $\ell$ }(, without specific initial state). 
\item And finally we let $n(a,\ell,m)$ be \emph{the number of different
nonadjacent products of order $n$, length $l$ with initial state
$a$}.
\end{enumerate}
\end{defn}

\begin{rem*}
By convention, $N(\ell,n)$ and $n(a,\ell,m)$ can both be zero, if
no such product exist. 
\end{rem*}
\begin{proof}
Let $\ell$ be an integer such that $3\leq\ell\leq n$, then
\begin{equation}
N(\ell,n)=N(\ell,n-1)+N(\ell-1,n-2).\label{eq:NlnRec}
\end{equation}
All the nonadjacent products of order $n$ and length $\ell$ can
be divide into two part: the first part consists of the existed nonadjacent
products before adding $a_{n}$, which are all nonadjacent products
of order $n$ and length $\ell$; and the second part consists of
the new nonadjacent products after adding $a_{n}$, which are all
nonadjacent products of order $n-2$ and length $\ell-1$. 
\end{proof}
\begin{thm}
Suppose $2\le\ell\le n$, then 
\[
n(1,\ell,n)=\sum_{k=3}^{n}n(k,\ell-1,n)=N(\ell-1,n-2).
\]
\end{thm}

\begin{proof}
A nonadjacent product of order $n$, length $\ell$ and initial state
$1$ starts with $a_{1}$ and follows with $a_{j},j\ge3$. So the
number of all nonadjacent products of order $n$, length $\ell$ and
initial state $1$ starts with $a_{1}$ equals to the number of all
nonadjacent products of order $n$, length $\ell-1$ and initial state
$1$ starts with $a_{j},j\ge3$, which also equals to the number of
all nonadjacent products of order $n-2$ and length $\ell-1$. 
\end{proof}
Note that $\binom{n}{k}=0$ if $k>n$, or $k<0$. So we can identify
$N(\ell,n)$ as the binomial coefficients. 
\begin{thm}
Suppose $1\le\ell\le n$, then 
\[
N(\ell,n)=\binom{n-\ell+1}{\ell}.
\]
\end{thm}

\begin{proof}
It suffices to show that $\binom{n-l+1}{l}$ satisfies the same initial
conditions and recurrence relation with $N(\ell,n)$, which is easy
to see, since 
\[
\binom{n-\ell+1}{\ell}=\binom{n-\ell}{\ell}+\binom{n-\ell}{\ell-1}.
\]
coincides with (\ref{eq:NlnRec}). For the initial conditions, notice
that when $\ell=1$, 
\[
N(1,n)=n=\binom{n}{1},
\]
which are exact the loops $L_{1},\ldots,L_{n}$. 
\end{proof}
Recall the \emph{multinomial coefficients}: for $k_{1},\ldots,k_{m}\in\mathbb{N}$,
$k_{1}+\cdots+k_{m}\leq n$, 
\[
\binom{n}{k_{1},\ldots,k_{m}}=\frac{n!}{k_{1}!\cdots k_{m}!(n-k_{1}-\cdots-k_{m})!}.
\]
In particular, for $m=1$, we have the binomial coefficients $\binom{n}{k}$. 
\begin{thm}
Let $M=\lceil m\rceil-1$, and $M'$ be the largest odd number less
or equal to $M$, then 
\begin{align}
E_{n}\left(\frac{x}{m+1}\right) & =\frac{1}{(m+1)^{n}}\sum_{k=0}^{\infty}\frac{(m+1)^{k}}{2^{2k+m}}\sum_{n_{1},\dots,n_{M}=0}^{k}\binom{k}{n_{1},\dots,n_{M}}\label{eq:1dimIdentity}\\
 & \ \ \ \times(-1)^{n_{1}+n_{3}+\cdots+n_{M'}}\frac{4^{n_{1}+\cdots+n_{M}}}{(m+1)^{n_{1}+\cdots+n_{M}}}\nonumber \\
 & \ \ \ \times\left(\frac{\binom{m-2}{1}}{2^{3}}+\frac{\binom{m-2}{2}}{2^{4}}\right)^{n_{1}}\cdots\left(\frac{\binom{m-M-1}{M}}{2^{2M+1}}+\frac{\binom{m-M-1}{M+1}}{2^{2M+2}}\right)^{n_{M}}\nonumber \\
 & \ \ \ \times E_{n}^{(2k+2n_{1}+4n_{2}+\cdots+2Mn_{M}+m)}\left(k+n_{1}+2n_{2}+\cdots+Mn_{M}+x\right).\nonumber 
\end{align}
\end{thm}

\begin{proof}
Let $n=m$ in (\ref{eq:Main}) and apply (\ref{eq:1dim}) to have
\begin{align*}
\sech((m+1)w) & =\sech w\frac{\sinh^{m}w}{\sinh^{m}(2w)}\sum_{k=0}^{\infty}\left(n(1,1,m)\frac{\sech^{2}w}{2}+\right.\\
& \quad N(1,m-1)\frac{\sech^{2}w}{2^{2}}-n(1,2,m)\frac{\sech^{4}w}{2^{3}}-\\
& \quad N(2,m-1)\frac{\sech^{4}w}{2^{4}}+\cdots+(-1)^{M}n(1,M+1,m)\times\\
& \quad\left.\frac{\sech^{(2M+2)}w}{2^{2M+1}}+(-1)^{M}N(M+1,m-1)\frac{\sech^{(2M+2)}w}{2^{2M+2}}\right)^{k}\\
& =\frac{\sech^{(m+1)}w}{2^{m}}\sum_{k=0}^{\infty}\left(\frac{\sech^{2}w}{2}+\frac{\binom{m-1}{1}\sech^{2}w}{2^{2}}-\right.\\
& \quad\frac{\binom{m-2}{1}\sech^{4}w}{2^{3}}-\frac{\binom{m-2}{2}\sech^{4}w}{2^{4}}+\cdots+(-1)^{M}\times\\
& \quad\left.\frac{\binom{m-M-1}{M}\sech^{(2M+2)}w}{2^{2M+1}}+(-1)^{M}\frac{\binom{m-M-1}{M+1}\sech^{(2M+2)}w}{2^{2M+2}}\right)^{k}.
\end{align*}
Applying (\ref{eq:Ep}) to get 
\begin{align*}
e^{(m+1)w(2\mathcal{E}+1)} & =\sum_{k=0}^{\infty}\frac{(m+1)^{k}}{2^{2k+m}}\sum_{n_{1},\dots,n_{M}=0}^{k}\binom{k}{n_{1},\dots,n_{M}}\frac{(-1)^{n_{1}+n_{3}+\cdots+n_{M'}}4^{n_{1}+\cdots+n_{M}}}{(m+1)^{n_{1}+\cdots+n_{M}}}\\
& \quad\times\left(\frac{\binom{m-2}{1}}{2^{3}}+\frac{\binom{m-2}{2}}{2^{4}}\right)^{n_{1}}\cdots\left(\frac{\binom{m-M-1}{M}}{2^{2M+1}}+\frac{\binom{m-M-1}{M+1}}{2^{2M+2}}\right)^{n_{M}}\\
& \quad\times e^{w(2\mathcal{E}^{(2k+2n_{1}+4n_{2}+\cdots+2Mn_{M}+m+1)}+2k+2n_{1}+4n_{2}+\cdots+2Mn_{M}+m+1)}.
\end{align*}
Multiplying by $e^{wx}$ produces 
\begin{align*}
e^{w(2m\mathcal{E}+2\mathcal{E}+x)} & =\sum_{k=0}^{\infty}\frac{(m+1)^{k}}{2^{2k+m}}\sum_{n_{1},\dots,n_{M}=0}^{k}\binom{k}{n_{1},\dots,n_{M}}\frac{(-1)^{n_{1}+n_{3}+\cdots+n_{M'}}4^{n_{1}+\cdots+n_{M}}}{(m+1)^{n_{1}+\cdots+n_{M}}}\\
& \quad\times\left(\frac{\binom{m-2}{1}}{2^{3}}+\frac{\binom{m-2}{2}}{2^{4}}\right)^{n_{1}}\cdots\left(\frac{\binom{m-M-1}{M}}{2^{2M+1}}+\frac{\binom{m-M-1}{M+1}}{2^{2M+2}}\right)^{n_{M}}\\
& \quad\times e^{w(2\mathcal{E}^{(2k+2n_{1}+4n_{2}+\cdots+2Mn_{M}+m+1)}+2k+2n_{1}+4n_{2}+\cdots+2Mn_{M}+x)}.
\end{align*}
Now, we identify the coefficients of $w^{n}$ on both sides: 
\begin{itemize}
\item The left-hand side is simply
\[
(2m\mathcal{E}+2\mathcal{E}+x)^{n}=(2m+2)^{n}\left(\mathcal{E}+\frac{x}{2m+2}\right)^{n}=(2m+2)^{n}E_{n}\left(\frac{x}{2m+2}\right);
\]
\item while the right-hand side, we only need to focus on the term 
\begin{align*}
& \quad(2\mathcal{E}^{(2k+2n_{1}+4n_{2}+\cdots+2Mn_{M}+m)}+2k+2n_{1}+4n_{2}+\cdots+2Mn_{M}+x)^{n}\\
& =2^{n}E_{n}^{(2k+2n_{1}+4n_{2}+\cdots+2Mn_{M}+m)}\left(k+n_{1}+2n_{2}+\cdots+Mn_{M}+\frac{x}{2}\right).
\end{align*}
\end{itemize}
Therefore, simplification, with $x\mapsto2x$, gives the desired identity. 
\end{proof}
\begin{example}
The formulas derived from $4$- and $5$-loop cases are as follows.
\begin{align*}
E_{n}\left(\frac{x}{5}\right) & =\frac{1}{5^{n}}\sum_{k=0}^{\infty}\sum_{\ell=0}^{k}\frac{5^{k}(-1)^{\ell}}{2^{2k+2\ell+4}}\binom{k}{\ell}E_{n}^{(2\ell+2k+5)}\left(x+\ell+k\right),\\
E_{n}\left(\frac{x}{6}\right) & =\frac{1}{6^{n}}\sum_{k=0}^{\infty}\sum_{n_{1},n_{2}=0}^{k}\frac{(-1)^{n_{1}}3^{k-n_{1}-n_{2}}}{2^{k+3n_{1}+4n_{2}+5}}\binom{k}{n_{1},n_{2}}\\
 & \quad\times E_{n}^{(2k+2n_{1}+4n_{2}+6)}\left(x+k+n_{1}+2n_{2}\right).
\end{align*}
\end{example}

\subsection{$3$-dim Bessel process on $\mathbb{R}^{3}$. }

\begin{figure}
\includegraphics[scale=0.3]{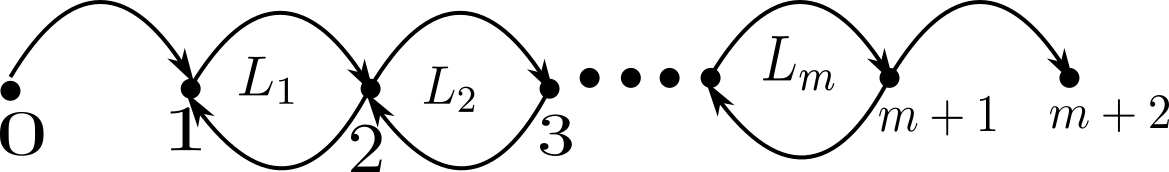}

\caption{\label{fig:3dim}$3$-dim Bessel Process}
\end{figure}

We can also find the generating functions for three consecutive sites
$a<b<c$, \cite[pp.~463--464]{Formulas} with variable $w$: 
\begin{align*}
\phi_{a\rightarrow b} & =\frac{b\sinh(aw)}{a\sinh(bw)},\\
\phi_{b\rightarrow a|\cancel{c}} & =\frac{a\sinh\left((c-b)w\right)}{b\sinh\left((c-a)w\right)},\\
\phi_{b\rightarrow c|\cancel{a}} & =\frac{c\sinh\left((b-a)w\right)}{b\sinh\left((c-a)w\right)}.
\end{align*}
As stated in \cite[Rem.~4.1]{JiuVignat}, $\phi_{a\rightarrow0|\cancel{b}}=0$
for $0<a<b$. In this case, the first loop occurs between site $1$
and $2$ (instead of $0$ and $1$). We can still have (\ref{eq:Main})
by shifting all the indices as $a_{m}\mapsto m+1$, ONLY for the loops: 
\begin{thm}
For the $3$-dimensional Bessel process on sites $0,1,\ldots,m+2$,
we have 
\begin{equation}
\phi_{0\rightarrow(m+2)}=\phi_{0\rightarrow1}\prod_{j=1}^{m+1}\phi_{j\rightarrow(j+1)|\cancel{j-1}}\frac{1}{1-\sum\limits _{k=1}^{m}\sum\limits _{(k,l,n)}}.\label{eq:3dim}
\end{equation}
Let $M=\lceil m\rceil-1$, and $M'$ be the largest odd number less
or equal to $M$, then 
\begin{align*}
& \quad B_{n+1}\left(\frac{2+x}{m+2}\right)-B_{n+1}\left(\frac{x}{m+2}\right)\\
& =\frac{n+1}{(m+2)^{n}}\sum_{k=0}^{\infty}\frac{m^{k}}{2^{2k+m}}\sum_{n_{1},\dots,n_{M}=0}^{k}\binom{k}{n_{1},\dots,n_{M}}(-1)^{n_{1}+n_{3}+\cdots+n_{M'}}\\
& \quad\times\frac{\binom{m-1}{2}^{n_{1}}\binom{m-2}{3}^{n_{2}}\cdots\binom{m-M}{M}^{n_{M}}}{4^{n_{1}+2n_{2}+\cdots+Mn_{M}}m^{n_{1}+\cdots+n_{M}}}\\
& \quad\times E_{n}^{(2k+2n_{1}+4n_{2}+\cdots+2Mn_{M}+m)}(k+n_{1}+2n_{2}+\cdots+Mn_{M}+x).
\end{align*}
\end{thm}

\begin{proof}
The steps will be similar to that of (\ref{eq:1dimIdentity}), so
we shall skip some direct but tedious calculation steps. By (\ref{eq:3dim})
and the well-known formula $\sinh(2w)=2\sinh w\cosh w$, we have 
\begin{align*}
\frac{(m+2)w}{\sinh((m+2)w)} & =\frac{(m+2)w}{\sinh(2w)}\sum_{k=0}^{\infty}\left(\frac{\binom{m}{1}\sech^{2}w}{4}-\frac{\binom{m-1}{2}\sech^{4}w}{4^{2}}\right.\\
& \quad\left.+\cdots+(-1)^{M}\frac{\binom{m-M}{M+1}\sech^{2M+2}w}{4^{M+1}}\right)^{k}.
\end{align*}
By (\ref{eq:Bp}) and (\ref{eq:Ep}), we deduce that 
\begin{align*}
e^{(m+2)w(2\mathcal{B}+1)} & =e^{2w(2\mathcal{B}'+1)}\sum_{k=0}^{\infty}\frac{(m+2)m^{k}}{2^{2k+m+1}}\sum_{n_{1},\dots,n_{M}=0}^{k}\binom{k}{n_{1},\dots,n_{M}}\\
& \quad\times(-1)^{n_{1}+n_{3}+\cdots+n_{M'}}\frac{\binom{m-1}{2}^{n_{1}}\binom{m-2}{3}^{n_{2}}\cdots\binom{m-M}{M}^{n_{M}}}{4^{n_{1}+2n_{2}+\cdots+Mn_{M}}m^{n_{1}+\cdots+n_{M}}}\\
& \quad\times e^{w\left(2\mathcal{E}^{(2k+2n_{1}+4n_{2}+\cdots+2Mn_{M}+m)}+2k+2n_{1}+4n_{2}+\cdots+2Mn_{M}+m\right)}.
\end{align*}
In order to cancel the $e^{4w\mathcal{B}'}$ on the right-hand side,
we multiplying by $e^{4w\mathcal{U}+wx}$ to have 
\begin{align*}
e^{w((2m+4)\mathcal{B}+4\mathcal{U}+x)} & =\sum_{k=0}^{\infty}\frac{(m+2)m^{k}}{2^{2k+m+1}}\sum_{n_{1},\dots,n_{M}=0}^{k}\binom{k}{n_{1},\dots,n_{M}}\times\\
& \quad(-1)^{n_{1}+n_{3}+\cdots+n_{M'}}\frac{\binom{m-1}{2}^{n_{1}}\binom{m-2}{3}^{n_{2}}\cdots\binom{m-M}{M}^{n_{M}}}{4^{n_{1}+2n_{2}+\cdots+Mn_{M}}m^{n_{1}+\cdots+n_{M}}}\times\\
& \quad e^{w(2\mathcal{E}^{(2k+2n_{1}+4n_{2}+\cdots+2Mn_{M}+m)}+2k+2n_{1}+4n_{2}+\cdots+2Mn_{M}+x)}.
\end{align*}
When identifying the coefficient of $w^{n}$, the right-hand side
directly gives the Euler polynomial of higher-orders; while the left-hand
side is, by (\ref{eq:RVU}), 
\begin{align*}
((2m+4)\mathcal{B}+4\mathcal{U}+x)^{n} & =\int_{0}^{1}((2m+4)\mathcal{B}+4u+x)^{n}du\\
& =\frac{((2m+4)\mathcal{B}+4u+x)^{n+1}}{4(n+1)}\bigg|_{u=0}^{u=1}\\
& =\frac{(2m+4)^{n+1}}{4(n+1)}\left(B_{n+1}\left(\frac{x+4}{2m+4}\right)-B_{n+1}\left(\frac{x}{2m+4}\right)\right).
\end{align*}
Simplification and substitution $x\mapsto2x$ complete the proof. 
\end{proof}
\begin{example}
The identities from \emph{three} and \emph{four} loops are given by
\begin{align*}
& \quad B_{n+1}\left(\frac{x+2}{5}\right)-B_{n+1}\left(\frac{x}{5}\right)\\
& =\frac{n+1}{5^{n}}\sum_{k=0}^{\infty}\frac{3^{k}}{2^{2k+3}}\sum_{\ell=0}^{k}\binom{k}{\ell}(-1)^{\ell}\frac{1}{12^{\ell}}E_{n}^{(2k+2\ell+3)}(k+\ell+x),
\end{align*}
and 
\begin{align*}
& \quad B_{n+1}\left(\frac{x+2}{6}\right)-B_{n+1}\left(\frac{x}{6}\right)\\
& =\frac{n+1}{6^{n}}\sum_{k=0}^{\infty}\frac{4^{k}}{2^{2k+4}}\sum_{\ell=0}^{k}\binom{k}{\ell}(-1)^{\ell}\frac{3^{\ell}}{4^{2\ell}}E_{n}^{(2k+2\ell+4)}(k+\ell+x).
\end{align*}
\end{example}
We shall briefly verify the identities in the example for the readers by direct calculation of the generating functions of Bernoulli and Euler polynomials. 
\begin{proof}
It suffices to show
\begin{align*}
& \quad \sum_{n=0}^{\infty}\left(B_{n+1}\left(\frac{x+2}{5}\right)-B_{n+1}\left(\frac{x}{5}\right)\right)\frac{t^{n+1}}{(n+1)!}\\
& =t\sum_{n=0}^{\infty}\sum_{k=0}^{\infty}\frac{3^{k}}{2^{2k+3}}\sum_{\ell=0}^{k}\binom{k}{\ell}(-1)^{\ell}\frac{1}{12^{\ell}}E_{n}^{(2k+2\ell+3)}(k+\ell+x)\frac{\left(\frac{t}{5}\right)^n}{n!},
\end{align*}
where the left-hand side is
\begin{align*}
& \quad \sum_{n=0}^{\infty}\left(B_{n+1}\left(\frac{x+2}{5}\right)-B_{n+1}\left(\frac{x}{5}\right)\right)\frac{t^{n+1}}{(n+1)!}\\
& =\sum_{n+1=0}^{\infty}\left(B_{n+1}\left(\frac{x+2}{5}\right)-B_{n+1}\left(\frac{x}{5}\right)\right)\frac{t^{n+1}}{(n+1)!}-B_0\left(\frac{x+2}{5}\right)+B_0\left(\frac{x}{5}\right)\\
& =\left(\frac{t}{e^t-1}\right)e^{\frac{x+2}{5}\cdot t}-\left(\frac{t}{e^t-1}\right)e^{\frac{x}{5}\cdot t}\\
& =\frac{t(e^{\frac{2t}{5}}-1)}{e^t-1}e^{\frac{tx}{5}},
\end{align*}
while the right-hand side is
\begin{align*}
& \quad t\sum_{n=0}^{\infty}\sum_{k=0}^{\infty}\frac{3^{k}}{2^{2k+3}}\sum_{\ell=0}^{k}\binom{k}{\ell}(-1)^{\ell}\frac{1}{12^{\ell}}E_{n}^{(2k+2\ell+3)}(k+\ell+x)\frac{\left(\frac{t}{5}\right)^n}{n!}\\
& =t\sum_{k=0}^{\infty}\frac{3^{k}}{2^{2k+3}}\sum_{\ell=0}^{k}\binom{k}{\ell}(-1)^{\ell}\frac{1}{12^{\ell}}\sum_{n=0}^{\infty}E_{n}^{(2k+2\ell+3)}(k+\ell+x)\frac{\left(\frac{t}{5}\right)^n}{n!}\\
& =t\sum_{k=0}^{\infty}\frac{3^{k}}{2^{2k+3}}\sum_{\ell=0}^{k}\binom{k}{\ell}(-1)^{\ell}\frac{1}{12^{\ell}}\left(\frac{2}{e^{\frac{t}{5}}+1}\right)^{2k+2\ell+3}e^{(k+\ell+x)\frac{t}{5}}\\
& =te^{\frac{tx}{5}}\left(\frac{1}{e^{\frac{t}{5}}+1}\right)^3\sum_{k=0}^{\infty}3^ke^{\frac{kt}{5}}\left(\frac{1}{e^{\frac{t}{5}}+1}\right)^{2k}\sum_{\ell=0}^{k}\binom{k}{\ell}(-1)^{\ell}\frac{1}{3^{\ell}}\left(\frac{1}{e^{\frac{t}{5}}+1}\right)^{2\ell}e^{\frac{\ell t}{5}}\\
& =te^{\frac{tx}{5}}\left(\frac{1}{e^{\frac{t}{5}}+1}\right)^3\sum_{k=0}^{\infty}3^ke^{\frac{kt}{5}}\left(\frac{1}{e^{\frac{t}{5}}+1}\right)^{2k}\left(1-\frac{1}{3}\left(\frac{1}{e^{\frac{t}{5}}+1}\right)^2e^{\frac{t}{5}}\right)^{k}\\
& =te^{\frac{tx}{5}}\left(\frac{1}{e^{\frac{t}{5}}+1}\right)^3\frac{1}{1-\frac{3e^{\frac{t}{5}}\left(e^{\frac{t}{5}}+1\right)^2-e^{\frac{2t}{5}}}{\left(e^{\frac{t}{5}}+1\right)^4}}\\
& =te^{\frac{tx}{5}}\frac{e^{\frac{t}{5}}+1}{e^{\frac{4t}{5}}+e^{\frac{3t}{5}}+e^{\frac{2t}{5}}+e^{\frac{t}{5}}+1}\\
& =\frac{t(e^{\frac{2t}{5}}-1)}{e^t-1}e^{\frac{tx}{5}}\qedhere
\end{align*}
\end{proof}

\end{document}